\documentclass[11pt]{article} 
\usepackage{amssymb}
\usepackage{amsmath}
\usepackage{amscd}
\usepackage{latexsym}
\usepackage[table]{xcolor}

\usepackage{booktabs}

\usepackage{tikz}
\usetikzlibrary{positioning}
\newdimen\nodeDist
\newtheorem{defin}{}
\newtheorem{saetze}[defin]{}
\newtheorem{conjec}[defin]{}
\newtheorem{lemmas}[defin]{}
\newtheorem{folger}[defin]{}
\newtheorem{bemerk}[defin]{}

\newenvironment{theorem}  {\begin{saetze}\it {\bf Theorem:}}{\end{saetze}}

\newenvironment{proof}    {{\it Proof}:}{{\hfill \fillbox \bigskip}}

\newcommand{\fillbox}{\mbox{$\bullet$}}

\renewcommand{\j}{\ol{j}}
\newcommand{\ol}{\overline}

\newcommand{\Z}{\mathbb Z}

\newcommand{\GAP}{\textsc{Gap}}
\newcommand{\LiePRing}{\textsc{LiePRing}}

\renewcommand{\Im}{\mathrm{Im}}

\newenvironment{items}{\begin{list}{$\alph{item})$}
{\labelwidth18pt \leftmargin18pt \topsep3pt \itemsep1pt \parsep0pt}}
{\end{list}}

\newcommand{\bulit}{\item[$\bullet$]}

\begin{document}

\title{Symbolic computation of Schur multipliers with
       an application to the groups of order dividing $p^6$}
\author{Bettina Eick and Taleea Jalaeeyan Ghorbanzadeh}
\date{\today}
\maketitle

\begin{abstract}
We describe an algorithm to compute the Schur multipliers of all nilpotent
Lie $p$-rings in the family defined by a symbolic nilpotent Lie $p$-ring. 
Symbolic nilpotent Lie $p$-rings can be used to describe the isomorphism 
types of $p$-groups of order $p^n$ for $n \leq 7$ and all primes $p \geq n$. 
We apply our algorithm to compute the Schur multipliers of all $p$-groups 
of order dividing $p^6$.
\end{abstract}

\section{Introduction}

A Lie ring is an additive abelian group with a multiplication, denoted by
$[.,.]$, that is bilinear, alternating and satisfies the Jacobi-identity.
A Lie $p$-ring is a Lie ring with $p^n$ elements for some prime power $p^n$.
A Lie $p$-ring is nilpotent, if its lower central series terminates at 
$\{0\}$. A nilpotent Lie $p$-ring $L$ of order $p^n$ can be described by a 
presentation $P(A)$ on $n$ generators $b_1, \ldots, b_n$ with coefficients 
$A = (a_{ijk}, a_{ik} \mid 1 \leq i < j < k \leq n)$, where $a_{ijk}$ and 
$a_{ik}$ in the field $\Z_p = \{ 0, \ldots, p-1 \}$, and relations 
\begin{eqnarray*}
  [b_i, b_j] &=& \sum_{k=j+1}^n a_{ijk} b_k 
        \;\; \mbox{ for } \;\; 1 \leq i<j \leq n, \\
   p b_i &=& \sum_{k=i+1}^n a_{ik} b_k 
        \;\; \mbox{ for } \;\; 1 \leq i \leq n.
\end{eqnarray*}

We generalize this type of presentation: we assume that $p$ 
is an indeterminate and the elements $a_{ijk}$ and $a_{ik}$ are polynomials 
in a polynomial ring $\Z[w, x_1, \ldots, x_m]$. Then $P(A)$ is a {\em 
symbolic} nilpotent presentation with respect to a set of primes $\Pi$ if 
for each $p \in \Pi$ and each $x_1, \ldots, x_m \in \Z_p$ the presentation
$P(A)$ evaluated at these values and $w$ being a primitive root mod $p$ 
defines a nilpotent Lie $p$-ring of order $p^n$. In this case we say that 
$P(A)$ defines a {\em symbolic} nilpotent Lie $p$-ring. 

The first aim here is to describe an algorithm to compute the Schur 
multipliers of the nilpotent Lie $p$-rings described by a symbolic 
nilpotent Lie $p$-ring. This algorithm translates, modifies and 
generalizes the method to compute the Schur multiplier of a polycyclic 
group as introduced by Eick \& Nickel \cite{ENi06}. 

Newman, O'Brien \& Vaughan-Lee \cite{NOV04, OVL05} determined a 
classification up to 
isomorphism of the nilpotent Lie $p$-rings of order dividing $p^7$. This 
is available in the LiePRing package \cite{liepring} of GAP \cite{GAP}
in the form of finitely many symbolic nilpotent presentations for
Lie $p$-rings. We apply our algorithm to this classification and hence
obtain a complete list of the Schur multipliers of the nilpotent Lie 
$p$-rings of order dividing $p^6$. The following table lists for each
$p^n$ with $1 \leq n \leq 6$ and all primes $p \geq 5$ the number 
of isomorphism types of nilpotent Lie $p$-rings of order $p^n$ with 
Schur multiplier isomorphic to $M$. A similar result for the primes 
$2$ and $3$ can be computed readily and is included in Section \ref{res} 
below. In the following table we use $A:=\gcd(p-1,3)$, $B:=\gcd(p-1,4)$ 
and $C:=\gcd(p-1,5)$.

\[\def\arraystretch{1.3}
\rowcolors{2}{gray!10}{white}
\begin{array}{lcccccc}
\toprule
$M$          & p^1 & p^2 & p^3 & p^4 & p^5 & p^6 \\
\midrule
\{0\}                & 1   & 1   & 2   & 2   & p+4 & p+3 \\
\Z_p                 &     & 1   & 1   & 5   & p+2A+B+14 & 6p+4A+B+18 \\
\Z_p^2               &     &     & 1   & 4   & 9   & 3p^2+19p+14A+7B+2C+42  \\
\Z_p^3               &     &     & 1   & 1   & 17  & 9p+6A+3B+106\\
\Z_p^4               &     &     &     & 1   & 4   & 4p+39  \\
\Z_p^5               &     &     &     &     & 4   & 52 \\
\Z_p^6               &     &     &     & 1   & 2   & 32 \\
\Z_p^7               &     &     &     &     & 1   & 4 \\
\Z_p^8               &     &     &     &     &     & 4 \\
\Z_p^9               &     &     &     &     &     & 6 \\
\Z_p^{10}            &     &     &     &     & 1   & 1 \\
\Z_p^{11}            &     &     &     &     &     & 1 \\
\Z_p^{15}            &     &     &     &     &     & 1 \\
\Z_{p^2}             &     &     &     & 1   & 3   & 6 \\
\Z_p + \Z_{p^2}      &     &     &     &     & 1   & 9 \\
\Z_p^2 + \Z_{p^2}    &     &     &     &     & 1   & 8 \\
\Z_p^3 + \Z_{p^2}    &     &     &     &     &     & 6 \\
\Z_p^4 + \Z_{p^2}    &     &     &     &     &     & 1 \\
\Z_p^5 + \Z_{p^2}    &     &     &     &     &     & 1 \\
\Z_{p^2}^2           &     &     &     &     &     & 1 \\
\Z_p + \Z_{p^2}^2    &     &     &     &     &     & 1 \\
\Z_{p^2}^3           &     &     &     &     &     & 1 \\
\Z_{p^3}             &     &     &     &     &     & 1 \\
\bottomrule
\end{array}
\]

The Lazard correspondence, see \cite{Laz54}, establishes an equivalence 
between the categories of nilpotent Lie $p$-rings of class less than $p$ 
and of $p$-groups of class less than $p$. This associates to any nilpotent 
Lie $p$-ring $L$ of class less than $p$ a $p$-group $G(L)$ of the same 
class. The morphisms in both categories are isomorphisms and thus a
classification up to isomorphism of the nilpotent Lie $p$-rings 
of order $p^n$ translates to the classification of groups of order $p^n$ 
for all $p \geq n$. 
In \cite{EHZ12} it is observed that the Schur multiplier $M(L)$ of a 
nilpotent Lie $p$-ring of class less than $p$ satisfies $M(L) = M(G(L))$.
Hence the above table is also a table of the Schur multipliers of the 
$p$-groups of order dividing $p^6$ for all primes $p \geq 5$.

The Schur multipliers of the groups of order dividing $p^5$ have been
determined in \cite{HKY18} using ad-hoc methods and in the unpublished
work \cite{HZa16} using an algorithmic approach. It was one motivation
for this work to exhibit a better algorithmic method that would also cover 
the groups of order $p^6$. 
\bigskip

{\bf Acknowledgement} We thank Max Horn for various discussions and reading 
and commenting an earlier draft of this work. 

\section{Preliminaries}

First, we recall that the multiplication of a Lie ring is
\begin{items}
\bulit
{\bf alternating:} $[x,x] = 0$ for $x \in L$,
\bulit
{\bf antisymmetric:} $[x,y] = -[y,x]$ for $x,y \in L$,
\bulit
{\bf bilinear}
$[x+y,z] = [x,z] + [y,z]$ and $[x,y+z] = [x,y] + [x,z]$,
for $x,y,z \in L$, and 
\bulit
{\bf Jacobi-identity}
$[x,[y,z]] + [z,[x,y]] + [y,[z,x]] = 0$
for $x,y,z \in L$.
\end{items}

The Schur multiplier of a group $G$ can be defined as the second homology 
group $H_2(G,\Z)$. For a Lie ring $L$, the Schur multiplier is defined as 
the second Chevalley-Eilenberg homology group of $L$. We recall the
following properties of Schur multipliers of nilpotent Lie $p$-rings from
\cite{EHZ12}.

\begin{theorem}
\label{schur}
Let $L$ be a nilpotent Lie $p$-ring and let $L = F/R$ for a free Lie ring
$F$.
\begin{items}
\bulit 
$M(L)$ is a finite abelian $p$-group.
\bulit
$M(L) \cong ([F,F] \cap R)/[F,R]$.
\bulit
$([F,F] \cap R)/[F,R]$ is the torsion subgroup of $R/[F,R]$.
\end{items}
\end{theorem}

\section{The Schur multiplier of a nilpotent Lie $p$-ring}
\label{algo1}

Let $L$ be a nilpotent Lie $p$-ring defined by a presentation $P(A)$.
Let $V_n$ denote the vector space of dimension $n(n+1)/2$ over $\Z_p$ 
with basis
\[B_n = \{ t_{ij} \mid 1 \leq i <  j \leq n\} \cup \{s_i \mid 1 \leq
i \leq n\}.\]
Define $t_{ii} = 0$ and $t_{ij} = -t_{ji}$ for $j < i$. Further, we 
introduce the following elements of $V_n$ for $1 \leq i, j, h \leq n$:
\begin{items}
\bulit
$u_i = \sum_{k=1}^n a_{ik} t_{ki}$,
\bulit
$v_{ij} = p t_{ij} + \sum_{k=1}^n (a_{ijk} s_k - a_{ik} t_{kj})$ 
     for $i \neq j$, and
\bulit
$w_{ijh} = \sum_{k=1}^n (a_{jhk} t_{ik} + a_{ijk} t_{hk} + a_{hik} t_{jk})$
     for $i < j < h$.
\end{items}

Let $Mat(A)$ denote the matrix whose rows consist of the coefficient vectors 
of the elements $u_i, v_{ij}, w_{ijh}$ with respect to the basis $B_n$.

\begin{theorem}
\label{schumu}
Let $P(A)$ define a nilpotent Lie $p$-ring of order $p^n$. Then 
the abelian invariants of the Schur muliplier $M(P(A))$ coincide 
with the the non-zero elementary divisors of $Mat(A)$.
\end{theorem}

\begin{proof}
Let $F$ be the free Lie ring on $b_1, \ldots, b_n$ and write $P(A) = 
F/R$ for $R$ generated by the relations of $P(A)$. Our aim is to show 
that $R/[R,F] \cong \Z^{n(n+1)/2} / \Im(Mat(A))$. This yields the 
desired claim, as $M(P(A))$ is isomorphic to the torsion subgroup of 
$R/[R,F]$ by Theorem \ref{schur} and the non-zero elementary divisors 
of $Mat(A)$ describe exactly that, see Sims \cite[Sec. 3.8]{Sim94}.

Recall that $B_n = \{ t_{ij}, s_i \mid 1 \leq i < j \leq n \}$ has 
$n(n+1)/2$ elements. Let $L(A)$ denote the Lie ring generated by 
$\{ b_1, \ldots, b_n \} \cup B_n$ subject to the relations
\begin{eqnarray*}
[b_i, b_j] &=& \sum_{k=1}^n a_{ijk} b_k + t_{ij} \mbox{ for } i \neq j, \\
pb_i &=& \sum_{k=1}^n a_{ik} b_k + s_i, \\
t_{ij} \mbox{ and } s_i && \mbox{ central}.
\end{eqnarray*}
Let $\langle B_n \rangle$ denote the subring generated by $B_n$ in $L(A)$. 
Then similarly to Lemma 1 of \cite{ENi06} it follows that $L(A) \cong 
F/[R,F]$ and $\langle B_n \rangle \cong R/[R,F]$. 

We now determine a presentation for the abelian subgroup $\langle B_n 
\rangle$ of $L(A)$. Similar as in \cite{ENi06} this can be obtained by
evaluating 'consistency relations' in the generators $b_1, \ldots, b_n$.
These consistency relations in $L(A)$ are:
\begin{items}
\bulit
$[p b_i, b_i] = p [b_i, b_i] = 0$ for all $i$.
\bulit
$[p b_i, b_j] = p [b_i, b_j]$ for all $i \neq j$.
\bulit
$[b_i, [b_j, b_h]] + [b_h, [b_i, b_j]] + [b_j, [b_h, b_i]] = 0$ 
for all $i < j < h$.
\end{items}
We note that all these relations hold in $P(A)$, since $P(A)$ defines
a nilpotent Lie $p$-ring of order $p^n$ and hence is consistent. We 
evaluate in $L(A)$ for all $i,j,h$ and $i \neq \j$:
\begin{eqnarray*}
p [b_i, b_{\j}]
  &=& p \left(\sum_{k=1}^n a_{i \j k} b_k + t_{i \j }\right) \\
  &=& \sum_{k=1}^n a_{i \j k} \,p b_k + p t_{i \j } \\
  &=& \sum_{k=1}^n a_{i \j k} \left(\sum_{l=1}^n a_{kl} b_l + s_k\right) 
          + p t_{i \j }  \\
  &=& \sum_{l=1}^n \left(\sum_{k=1}^n a_{i \j k} a_{kl}\right) b_l 
          + \sum_{k=1}^n a_{i \j k} s_k + p t_{i \j } \\
\end{eqnarray*}
\begin{eqnarray*}
[p b_i, b_j] 
  &=& \left[ \sum_{k=1}^n a_{ik} b_k + s_i,\ b_j\right] \\
  &=& \sum_{k=1}^n a_{ik} [b_k, b_j]\\
  &=& \sum_{k=1}^n a_{ik} \left(\sum_{l=1}^n a_{kjl} b_l + t_{kj}\right) \\
  &=& \sum_{l=1}^n \left(\sum_{k=1}^n a_{ik} a_{kjl}\right) b_l 
       + \sum_{k=1}^n a_{ik} t_{kj} \\
\end{eqnarray*}
\begin{eqnarray*}
[b_i, [b_j, b_h]]
   &=& \left[b_i, \sum_{k=1}^n a_{jhk} b_k + t_{jh}\right] \\
   &=& \sum_{k=1}^n a_{jhk} [b_i, b_k] \\
   &=& \sum_{k=1}^n a_{jhk} \left(\sum_{l=1}^n a_{ikl} b_l + t_{ik}\right) \\
   &=& \sum_{l=1}^n \left(\sum_{k=1}^n a_{jhk} a_{ikl}\right) b_l 
        + \sum_{k=1}^n a_{jhk} t_{ik}
\end{eqnarray*}

Using that the relations hold in $P(A)$ and using the definitions for 
$u_i, v_{ij}$ and $w_{ijh}$ it now follows that
\begin{items}
\bulit
$u_i = [pb_i, b_i]$, 
\bulit
$v_{ij} = p[b_i,b_j] - [pb_i,b_j]$, and 
\bulit
$w_{ijh} = [b_j, [b_j,b_h]] + [b_j, [b_h,b_i]] + [b_h, [b_i,b_j]]$.
\end{items}
Hence $u_i = v_{ij} = w_{ijh} = 0$ in $L(A)$ and it follows that 
$R/[R,F] \cong \Z^{n(n+1)/2} / \Im(Mat(A))$ as desired.
\end{proof}

\section{Schur multipliers for symbolic nilpotent Lie $p$-rings}
\label{algo2}

Let $P(A)$ be a symbolic nilpotent Lie $p$-ring; that is, the prime
$p$ of $P(A)$ is indeterminate and the elements $a_{ijk}$ and $a_{ij}$
of $A$ are integral polynomials in the indetermines $w, x_1, \ldots, x_m$.

Then $P(A)$ defines a family of nilpotent Lie $p$-rings of order $p^n$.
Our aims are to determine the Schur multipliers 
for {\em all} members of the family simultaneously. The principal
approach towards this aim is very similar to the algorithm described
in Section \ref{algo1}. 

Given $P(A)$, we first determine the matrix $Mat(A)$. This can be read 
off readily from $A$ and it is a matrix with entries in 
$\Z[p,w,x_1, \ldots, x_m]$. Our aim now reduces to computing the
elementary divisors of $Mat(A)$ for all possible evaluations of the
indeterminates simulaneously.

We recall the computation of elementary divisors here briefly. Elementary
row (column) operations of a matrix are:
\begin{items}
\bulit
Interchanging two rows (columns),
\bulit
Adding a multiple of one row (column) to another,
\bulit
Multiplying one row (column) by a unit in the underlying ring.
\end{items}

Let $S = \Z[w, x_1, \ldots, x_m]$ and let $Q = Quot(S)$ the set of rational 
functions in $w, x_1, \ldots, x_m$ over $\Z$. An element in $Q[p]$ has the
form $s_n p^n + \ldots + s_1 p + s_0$ with $s_n, \ldots, s_0 \in Q$. We 
call such an element a {\em pseudo-unit} if $s_0 \neq 0$. Note that every 
non-zero element in $Q[p]$ can be written as $p^l u$ for a pseudo-unit $u$. 

We now perform a Smith normal form computation on $Mat(A)$ by allowing
elements in $Q[p]$ as entries in the matrix and by using pseudo-units
in the third case of elementary operations. Whenever we divide by a
pseudo-unit in the algorithm, we store the used pseudo-unit and return
the resulting list of pseudo-units together with the elementary divisors
of $Mat(A)$. Note that the elementary divisors are all powers of $p$ as
a result.

A pseudo-unit $u$ is a rational function over $S = \Z[w, x_1, \ldots, x_m]$.
Let $E_u = \{ (x_1, \ldots, x_m) \in \Z_p^m \mid u(w, x_1, \ldots, x_m) = 0\}$
the (finite) set of zeros of $u$. Then for all elements outside $E_u$ the
pseudo-unit $u$ yields a proper unit in $\Z_p$. Hence for all elements 
outside 
\[ E = \bigcup_{u \mbox{ pseudo-unit}} E_u \]
we can determine the Schur multiplier of $P(A)$ from $Mat(A)$ via Theorem
\ref{schumu}. It remains to consider the finite set $E$. For each tuple
in $E$ we evaluate the presentation $P(A)$ in the tuple and then determine
the Schur multiplier with the method of the previous section in the 
evaluated presentation.

\section{Examples}

We exhibit two examples. Both refer to the library of Lie $p$-rings of
dimension $6$ in the \LiePRing\ package of \GAP. The first example gives
a first introduction into the algorithm. The second example is the most
difficult case for our algorithm among all Lie $p$-rings of dimension $6$.

\subsection{First example}

As a first example, we consider the parametrised Lie $p$-ring $L$ number 
$245$ of dimension $6$. This defines a family of $p-1$ Lie $p$-rings 
with parameter $x \in \{1, \ldots, p-1\}$. 

\begin{verbatim}
gap> L := LiePRingsByLibrary(6)[245];
<LiePRing of dimension 6 over prime p with parameters [ x ]>
gap> NumberOfLiePRingsInFamily(L);
p-1
gap> LibraryConditions(L);
[ "x ne 0", "" ]
\end{verbatim}

We now apply the algorithm of Section \ref{algo2}.

\begin{verbatim}
gap> LiePSchurMultGen(L);
rec( norm := [ p, p ], pseudounits := [ -x^2-x ] )
\end{verbatim}

The output asserts that $M(F) = \Z_p^2$ for all Lie $p$-rings $F$ in the
family defined by the symbolic nilpotent Lie $p$-ring $L$ with the possible
exception of $L$ evaluated at the places where $-x^2-x = 0$. Clearly,
$-x^2-x = 0$ if and only if $x \in \{0,-1\}$. The case $x=0$ is excluded
in the family defined by $L$. Hence it remains to consider the case 
$x = -1$. To consider this special case, we evaluate the Lie $p$-ring
presentation for $L$ at this place and then apply the algorithm of 
Section \ref{algo2} again.

\begin{verbatim}
gap> S := SpecialiseLiePRing(L, false, [x], [-1]);
<LiePRing of dimension 6 over prime p>
gap> LiePSchurMultGen(S);
rec( norm := [ p, p, p ], unit := [  ] )
\end{verbatim}

In summary, this parametrised Lie $p$-ring yields $p-2$ Lie $p$-rings 
with Schur multiplier $\Z_p^2$ and 1 Lie $p$-ring with Schur multiplier 
$\Z_p^3$.

\subsection{Second example}

The second example is the Lie p-ring $L$ with number $267$ of dimension $6$.
This has four parameters $x, y, z$ and $t$ and it defines a family containing
$(2p^2+p-(p-1,4)+1)/2$ Lie $p$-rings. 

\begin{verbatim}
gap> L := LiePRingsByLibrary(6)[267];
<LiePRing of dimension 6 over prime p with parameters [ w, x, y, z, t ]>
gap> NumberOfLiePRingsInFamily(L);
p^2+1/2*p-1/2*(p-1,4)+1/2
gap> LibraryConditions(L);
[ "See note 6.178", "" ]
\end{verbatim}

The LibraryConditions of a parametrised Lie $p$-ring specify which values
of the parameters yield a complete and irredundant set of isomorphism types 
of Lie $p$-rings in this family. In the considered case, the parameters 
$x,y,z,t$ are elements of $\{0, \ldots, p-1\}$ so that
\[ A=\left( \begin{array}{cc} t & x \\ y & z \end{array} \right) \]
is non-singular modulo $p$. Two such parameter matrices $A$ and 
$B$ define isomorphic algebras if and only if
\[ B=\frac{1}{\det P}PAP^{-1} \bmod p \]
for some matrix $P$ of the form
\[ \left( \begin{array}{ll} \alpha  & \beta  \\ e \omega \beta  & e \alpha 
\end{array} \right) \]
with $e = \pm 1$, $\omega$ is a primitive element modulo $p$ and $\alpha^2 
- \omega \beta^2 \neq 0$ modulo $p$. The set of all matrices $P$ forms a 
subgroup $U$ of $GL(2,p)$ of size $2(p^2-1)$. We apply the algorithm of
Section \ref{algo2} to $L$.

\begin{verbatim}
gap> LiePSchurMultGen(L);
rec( norm := [ p, p ], unit := [ t, y, x*y*z+x*y*t-z^2*t-z*t^2 ] )
\end{verbatim}

Hence the algorithm determines three pseudo-units in this case. One of
them is $t$. As noted by Vaughan-Lee, every orbit of $U$ on the set of 
parameter matrices contains a matrix with $t \in \{0,1\}$. Hence it 
is sufficient to consider the cases $t = 0$ and $t=1$.

\begin{verbatim}
gap> S:=SpecialiseLiePRing(L, false, [t], [0]);;
gap> LiePSchurMultGen(S);
rec( norm := [ p, p ], unit := [ x, -y*z ] )

gap> S := SpecialiseLiePRing(L, false, [t], [1]);;
gap> LiePSchurMultGen(S);
rec( norm := [ p, p ], unit := [ x, x*y*z+x*y-z^2-z ] )
\end{verbatim}

We now analyse the occuring sets of pseudo-units. As $A$ is non-singular,
the case $t=0$ implies $x,y \neq 0$. Thus the pseudo-units reduce to $z$
and it remains to consider $(t,z) = (0,0)$ as special case. If $t=1$, then
the pseudo-units are $\{x, (z+1)(xy-z) \}$. As $A$ is non-singular, it 
follows that $xy \neq z$. Hence it remains to consider the special cases
$(t,x) = (1,0)$ and $(t,z) = (1,-1)$.

\begin{verbatim}
gap> S := SpecialiseLiePRing(L, false, [t,z], [0,0]);;
gap> LiePSchurMultGen(S);
rec( norm := [ p, p, p ], unit := [ x ] )

gap> S := SpecialiseLiePRing(L, false, [t,x], [1,0]);;
gap> LiePSchurMultGen(S);
rec( norm := [ p, p ], unit := [ z, z+1 ] )

gap> S := SpecialiseLiePRing(L, false, [t,z], [1,-1]);;
gap> LiePSchurMultGen(S);
rec( norm := [ p, p, p ], unit := [ ] )
\end{verbatim}

If $(t,z)=(0,0)$, then $x \neq 0$, and if $(t,x) = (1,0)$, then $z
\neq 0$, as $A$ is invertible. The case $(t,x,z) = (1,0,-1)$ is
covered by the third case. Hence the remaining pseudo-units can all
be considered as proper units. 

Note that there are $p-1$ Lie $p$-rings in the family defined by $L$ 
with $(t,z) = (0,0)$ and $(p+1)/2$ with $(t,z) = (1,-1)$. In summary,
$L$ yields $(2p^2-2p-(p-1,4)+2)/2$ Lie $p$-rings with Schur multiplier 
$\Z_p^2$ and $(3p-1)/2$ Lie $p$-rings with Schur multiplier $\Z_p^3$. 
The following graph visualizes the case-distinctions in this example.
Cases that do not play a role due to library conditions are indicated
by ``N/A''.

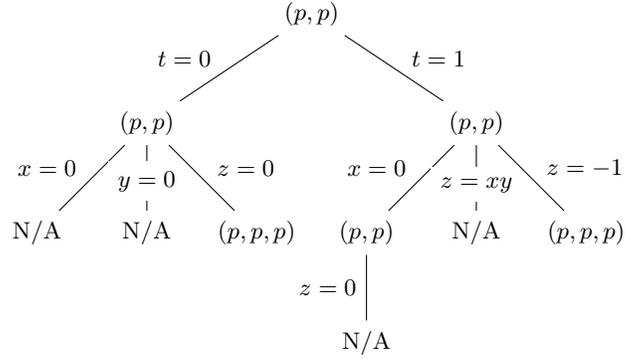
\begin{center}
\begin{tikzpicture}[scale=0.73,font=\footnotesize,
    node/.style={%
    },
  ]

    \node [node] (top) at (0,6) {$(p,p)$};
    \node [node] (t0) at (-3,4) {$(p,p)$};
    \node [node] (t1) at (+3,4) {$(p,p)$};

    \node [node] (t0 x0) at (-5,2) {N/A};
    \node [node] (t0 y0) at (-3,2) {N/A};
    \node [node] (t0 z0) at (-1,2) {$(p,p,p)$};

    \node [node] (t1 x0) at (+1,2) {$(p,p)$};
    \node [node] (t1 z=xy) at (+3,2) {N/A};
    \node [node] (t1 z=-1) at (+5,2) {$(p,p,p)$};

    \node [node] (t1 x0 z0) at (+1,0) {N/A};

    \draw (top)--(t0) node[left=0.3,pos=0.35]{$t=0$};
    \draw (top)--(t1) node[right=0.3,pos=0.35]{$t=1$};

    \draw (t0)--(t0 x0) node[left=0.2,pos=0.35]{$x=0$};
    \draw (t0)--(t0 y0) node[fill=white,pos=0.55]{$y=0$};
    \draw (t0)--(t0 z0) node[right=0.2,pos=0.35]{$z=0$};

    \draw (t1)--(t1 x0) node[left=0.2,pos=0.35]{$x=0$};
    \draw (t1)--(t1 z=xy) node[fill=white,pos=0.6]{$z=xy$};
    \draw (t1)--(t1 z=-1) node[right=0.2,pos=0.35]{$z=-1$};

    \draw (t1 x0)--(t1 x0 z0) node[left,pos=0.5]{$z=0$};
\end{tikzpicture}
\end{center}

\section{Limitations}
\label{limit}

Among the Lie $p$-rings of order $p^7$ in the LiePRing package there is
one with 13 parameters and the following presentation:

\begin{verbatim}
gap> ViewPCPresentation(L);
p*l1 = j*l5 + k*l6 + m*l7
p*l2 = n*l5 + r*l6 + s*l7
p*l3 = t*l5 + u*l6 + v*l7
p*l4 = x*l5 + y*l6 + z*l7
[l2,l1] = l5
[l3,l1] = l6
[l3,l2] = l7
[l4,l2] = w*l6
[l4,l3] = l5
\end{verbatim}

This Lie $p$-ring is currently beyond the range of our algorithm and it
is a main reasons why we limited our application to the groups of order 
dividing $p^6$.

\section{The groups of orders dividing $2^6$ and $3^6$}
\label{res}

The following table lists the numbers of groups of order $p^n$ with 
prescribed Schur multiplier for $n \leq 6$ and $p=2$ and $p=3$. We
only list those orders for which the result differs from the table in
the introduction. The information in this table can be computed readily
using GAP \cite{GAP} and we list it for completness only.

\[\def\arraystretch{1.3}\rowcolors{2}{gray!10}{white}
\begin{array}{lrrrrrr}
\toprule
M                  & 2^3 & 2^4 & 2^5 & 2^6& 3^5 & 3^6 \\
\midrule
\{0\}                & 2 & 4 & 5 &  9 & 5  &   6   \\ 
\Z_p                 & 2 & 3 & 12& 33 & 19 &  44  \\
\Z_p^2               & 1 & 3 & 14& 56 & 10 & 123  \\
\Z_p^3               &   & 2 & 6 & 41 & 14 & 132  \\
\Z_p^4               &   &   & 2 & 33 & 4  &  55  \\
\Z_p^5               &   &   & 4 & 33 & 4  &  44   \\ 
\Z_p^6               &   & 1 & 2 &  7 & 2  &  32   \\
\Z_p^7               &   &   &   &  2 & 1  &   4    \\ 
\Z_p^8               &   &   &   &    &    &   4    \\ 
\Z_p^9               &   &   &   &  5 &    &   6    \\
\Z_p^{10}            &   &   & 1 &  2 & 1  &   1    \\
\Z_p^{11}            &   &   &   &    &    &   1    \\
\Z_p^{15}            &   &   &   &  1 &    &   1    \\
\Z_{p^2}             &   & 1 & 1 & 10 & 4  &  15    \\
\Z_p + \Z_{p^2}      &   &   & 2 & 14 & 2  &  12    \\
\Z_p^2 + \Z_{p^2}    &   &   & 2 &  9 & 1  &  10    \\
\Z_p^3 + \Z_{p^2}    &   &   &   &  3 &    &   8    \\
\Z_p^4 + \Z_{p^2}    &   &   &   &  3 &    &   1    \\
\Z_p^5 + \Z_{p^2}    &   &   &   &  2 &    &   1    \\
\Z_{p^2}^2           &   &   &   &    &    &   1    \\
\Z_p + \Z_{p^2}^2    &   &   &   &  1 &    &   1    \\
\Z_{p^2}^3           &   &   &   &  1 &    &   1    \\
\Z_{p^3}             &   &   &   &  1 &    &   1    \\
\Z_p^2 + \Z_{p^2}^2  &   &   &   &  1 &    &        \\ 
\bottomrule
\end{array}
\]

\def\cprime{$'$} \def\cprime{$'$}

\end{document}